\documentclass{amsart}
\usepackage[T1]{fontenc}

\usepackage{amsthm}
\usepackage[leqno]{amsmath}
\usepackage{amssymb}
\usepackage{amsfonts}
\usepackage{color}

\newcommand{\rr}{\mathbf{\mathbb{R}}}
\newcommand{\cc}{\mathbf{\mathbb{C}}}
\newcommand{\nn}{\mathbf{\mathbb{N}}}

\newcommand{\ve}{\varepsilon}

\newcommand{\vf}{\varphi}

\newcommand{\grad}{\nabla}

\newcommand{\length}{\operatorname{length}}

\newcommand{\dist}{\operatorname{dist}}

\newcommand{\Int}{\operatorname{Int}}
\newcommand{\graph}{\operatorname{graph}}
\newcommand{\KS}{\operatorname{\mathcal K}}
\newcommand{\KSS}{\operatorname{\mathcal S}}

\theoremstyle{plain}
\newtheorem{twr}{Theorem}[section]

\newtheorem{lemat}[twr]{Lemma}

\newtheorem{cor}[twr]{Corollary}
\newtheorem{prop}[twr]{Proposition}
\newtheorem{fact}{Fact}

\numberwithin{equation}{section}

\theoremstyle{definition}

\newtheorem{remark}[twr]{Remark}
\newtheorem{exa}[twr]{Example}

\begin{document}

\title[Convexifying positive polynomials and s.o.s. approximation]{Convexifying positive polynomials\\ and {sums of squares} approximation}
\author{Krzysztof Kurdyka, Stanis{\l}aw Spodzieja}
\thanks{This research was partially supported by OPUS Grant No 2012/07/B/ST1/03293 (Poland) and by ANR Project STAAVF (France)}

\date{\today}
\begin{abstract} 
We show that if a polynomial $f\in \rr[x_1,\ldots,x_n]$ 
 is  nonnegative on a closed basic semialgebraic set $X=\{x\in\rr^n:g_1(x)\ge 0,\ldots,g_r (x)\ge 0\}$,\linebreak where $g_1,\ldots,g_r\in\rr[x_1,\ldots,x_n]$,  then $f$ can be approximated uniformly on compact sets by polynomials of the form $\sigma_0+\vf(g_1)  g_1+\cdots +\vf(g_r) g_r$, 
where $\sigma_0\in \rr[x_1,\ldots,x_n]$ and $\vf\in\rr[t]$ are sums of squares of polynomials. In particular, if $X$ is compact, and $h(x):=R^2-|x|^2 $ is positive on $X$, then $f=\sigma_{0}+\sigma_1  h+\vf(g_1) g_1+\cdots +\vf(g_r) g_r$ for some sums of squares 
$\sigma_{0},\sigma_1\in \rr[x_1,\ldots,x_n]$ and $\vf\in\rr[t]$,
 where $|x|^2={x_1^2+\cdots+x_n^2}$. 
 We apply  a quantitative version of those results to semidefinite optimization methods. Let $X$ be a convex closed semialgebraic subset of $\rr^n$ and let $f$ be a polynomial which is positive on $X$.  We give  necessary and sufficient conditions for the existence of  an exponent $N\in\mathbb{N}$ such that $(1+|x|^2)^Nf(x)$ is a convex function on $X$. We apply this result to  searching for lower  critical points of polynomials on convex  compact semialgebraic sets.
\end{abstract}
\maketitle

\section*{Introduction}

In the paper we study  two types of problems for  polynomials which are positive (or nonnegative) on   subsets of $\rr^n$.
In the first part  we prove  stronger versions of known approximation  and representation theorems  with sums of squares of polynomials. Next we give quantitative    versions  of  these
results and explain some applications  to semidefinite optimization methods. In the second part we prove that any  polynomial $f$ which is positive  on a convex closed set  $X$ becomes strongly convex when multiplied by $(1+|x|^2)^N$ with $N$ large enough (the noncompact case requires some extra assumptions).  In fact we give  an explicit estimate for $N$, which depends on the size of the coefficients of $f$ and on the lower bound of $f$ on~$X$.  As an application of our convexification method we propose an algorithm which for a given polynomial $f$ on a  compact semialgebraic set $X$  produces a sequence (starting from an arbitrary point in $X$) which converges to a critical point of $f$ on $X$.  We also relate convexity and positivity issues.

\subsection{Notation and state of the art}\label{notation}
We denote by $\rr[x]$ or $\rr[x_1,\ldots,x_n]$ the ring of polynomials in $x=(x_1,\ldots,x_n)$ with coefficients in $\rr$. 
Important  problems of real algebraic geometry are representations of nonnegative polynomials on closed semialgebraic sets.  
 Recall  Hilbert's 17th problem  (solved  by E.~Artin \cite{Artin}): if $f\in \rr[x]$ is nonnegative on $\rr^n$,  then 
 \begin{equation}\label{AH}\tag{AH}
fh^2=h_1^2+\cdots+h_m^2 \quad\hbox{for some $h,h_1,\ldots,h_m\in \rr[x]$, $h\ne 0$,}
\end{equation}
that is, $f$ is a sum of squares of rational functions. With the additional assumptions that $f$ is  homogeneous   and  $f(x)>0$ for $x\ne 0$, B.~Reznick \cite[Theorem 3.12]{Reznick} proved that there exists an integer $r_0$ such that for any $N\ge r_0$ the polynomial $(x_1^2+\cdots+x_n^2)^Nf(x)$ is a sum of even powers of linear functions.

Let $X\subset \rr^n$ be a \emph{closed basic semialgebraic set} defined by $g_1,\ldots,g_r\in \rr[x]$, i.e.,
\begin{equation}\label{semialgebraicbasicform}
X=\{x\in \rr^n:g_1(x) \ge 0,\ldots,g_r(x)\ge 0\}.
\end{equation}
The \emph{preordering} generated by $g_1,\ldots,g_r$ is defined to be
$$
T(g_1,\ldots,g_r)=\Big\{\sum_{e=(e_1,\ldots,e_r)\in\{0,1\}^r
}\sigma_e g_1^{e_1}\cdots g_r^{e_r}:\sigma_e\in \sum 
\rr[x]^2\hbox{ for }e\in\{0,1\}^r\Big\},
$$
where $\sum \rr[x]^2$ denotes the set of sums of squares (s.o.s.) of polynomials from $\rr[x]$. 
Natural generalizations of the above  theorem of Artin are the Stellens\"atze of \hbox{J.-L.~Krivine} \cite{Krivine}, D. W. Dubois \cite{Dubois}, and J.-J. Risler \cite{Risler} (see also \cite{BCRoy}). For references and a more detailed discussion of this subject see for instance \cite{Sr2}, \cite{Marshall}, \cite{PD}. When the set $X$ is compact, a very important  result   was  obtained by K.~Schm\"udgen (see \cite{Sn},  \cite{CMN}): every strictly positive polynomial $f$ on $X$ belongs to
the preordering  $T(g_1,\ldots,g_r)$. M.~Schweighofer \cite{Schweighofer1}  studied degree bounds in the Schm{\"u}dgen Positivstellensatz representation  
$$f=\sum_{e\in\{0,1\}^r}\sigma_e g_1^{e_1}\cdots g_r^{e_r}\in T(g_1,\ldots,g_r).$$ 
 He obtained an  upper bound  for $\deg \sigma_e g_1^{e_1}\cdots g_r^{e_r}$ in terms of $\deg f$,  $f^*:=\min \{f(x): x\in X\}$ and the coefficients of $f$, provided that
 $f^*>0$. As  shown by C.~Scheiderer \cite{Sr1}, there is no such   bound in terms of $\deg f$ unless $\dim (X) \le 1$. Under some additional assumptions M. Putinar \cite{Putinar} proved that $f$ belongs to the \emph{quadratic module} generated by $g_1,\ldots,g_r$,
$$
P(g_1,\ldots,g_r):=\Big\{\sigma_0+\sigma_1g_1+\cdots+\sigma_rg_r:\sigma_i\in\sum\rr[x]^2,\,i=0,\ldots,r\Big\}.
$$
The above results concern strictly positive polynomials. In the case of nonnegative polynomials C. Berg, J.~P.~R.~Chris\-ten\-sen and P. Ressel  \cite{BCR} and J.~B.~Lasserre and T.~Netzer \cite[Corollary 3.3]{ln} proved  that any  polynomial $f$ which is nonnegative on $[-1,1]^n$  can be approximated in the $l_1$-norm by sums of squares of polynomials.
 The $l_1$-norm of a polynomial is defined to be the sum of the absolute values
of its coefficients (in the usual monomial basis). Hence we have 

\begin{fact}\label{fact1}
 If a polynomial $f\in\rr[x]$ is nonnegative on $[- R,R]^n$, $R>0$, then the polynomial $f(R x)$ can be  approximated in the $l_1$-norm by sums of squares of polynomials. 
In particular $f(x)$ can be uniformly approximated on $[-R,R]^n$ by sums of squares of polynomials.
\end{fact}

D. Hilbert \cite{hilbert} proved  that   for $n\ge 2 $ there are nonnegative  polynomials on $\rr^n$ which are not sums of squares of polynomials.
T. S. Motzkin  \cite{Motzkin} gave an  explicit example of such a polynomial,  $f(x_1,x_2)=1+x_1^2x_2^2(x_1^2+x_2^2-3)$, i.e., 
 in the representation  \eqref{AH} of $f$ the degree of $h$ must be positive. So in general the Schm{\"u}dgen Positivstellensatz does not hold on noncompact sets. For  a polynomial $f$ positive on a noncompact set $X$ the problem arises of approximation of $f$ by elements of the preordering $T(g_1,\ldots,g_r)$ or of the quadratic module $P(g_1,\ldots,g_r)$. In this connection J.~B.~Lasserre 
\cite[Theorem 2.6]{las2} (see also \cite{Las0}) proved that if $g_1,\ldots,g_r$ are concave polynomials such that $g_1(z)>0,\ldots g_r(z)>0$ for some $z\in X$, then any convex polynomial nonnegative on $X$ can be approximated in the $l_1$-norm  by polynomials from the set
$$
L_c(g_1,\ldots,g_r):=\Big\{\sigma_0+\lambda^2_1g_1+\cdots+\lambda^2_rg_r:\sigma_0\in\sum\rr[x]^2\hbox{ convex, } \lambda_1,\ldots,\lambda_r\in\rr\Big\}.
$$
For $X=\rr^n$ the approximation is uniform on compact sets.  J. B. Lasserre \cite{Las0} proved that if a polynomial $f\in\rr[x]$ has a global
minimum $f^*\ge  0$ then for every $\varepsilon > 0$ there is $N\in\nn$ such that the polynomial 
$
f_\varepsilon := f +\varepsilon \sum_{k=1}^{N}\sum_{j=1}^n\frac{x_j^{2k}}{k!}
$ 
is a sum of squares (see also \cite{lr4} for polynomials on real algebraic sets).

\subsection{ Our contributions}\label{contribution}

In this article, we prove an analogue of the Schm\"udgen and Putinar theorems for a smaller cone.
Namely for   $g\in\rr[x]$ we put 
\begin{multline*}
\KS(g,g_1,\ldots,g_r):
=\Big\{\sigma_0+\sigma_1 g+\vf(g_1) g_1+\cdots +\vf(g_r)g_r: 
\sigma_0,\sigma_1\in\sum\rr[x]^2,\\ 
\vf\in\sum\rr[t]^2\Big\},
\end{multline*}
where  $t$ is a single variable. Note that if we set  
\[
\Phi(g_1,\ldots,g_r):=\left\{\vf(g_1) g_1+\cdots +\vf(g_r)g_r: \vf\in\sum\rr[t]^2\right\},
\]
then
\begin{equation*}
\KS(g,g_1,\ldots,g_r)=T(g)+\Phi(g_1,\ldots,g_r),
\end{equation*}
where $\mathcal{A}+\mathcal{B}=\{a+b:a\in \mathcal{A},\, b\in \mathcal{B}\}$. 
In Section~\ref{sectionapproximation1} we prove  (Theorem~\ref{approximation1}) 
that for a closed basic semialgebraic set $X$ defined by $g_1,\ldots,g_r\in \rr[x]$ and a polynomial $f\in\rr[x]$ 
the following conditions are equivalent:
\begin{enumerate}
\item[(i)]  $f$ is nonnegative on  $X$,
\item[(ii)]  $f$ can be uniformly approximated on compact sets by polynomials from the cone  
\[
\KSS(g_1,\ldots,g_r):=\sum\rr[x]^2+\Phi(g_1,\ldots,g_r).
\]
Moreover, $f$ can be approximated by polynomials from $\KSS(g_1,\ldots,g_r)$ in the $l_1$-norm. 
\end{enumerate}
 In particular, if $X$ is a compact set and $g(x):=R^2-|x|^2\ge 0$ for $x\in X$, then 
(see Corollary \ref{approximationPutinar})
\begin{equation}\label{equivalence1}
\hbox{$f$ is strictly positive on  $X$}\quad\Longrightarrow\quad\hbox{$f \in \KS(g,g_1,\ldots,g_r)$.}
\end{equation}

\subsection{ Application to optimization}\label{apploptsection}

In \cite{Las-1} Lasserre gave a method of minimizing a polynomial $f$ on a compact basic semialgebraic set $X$ of the  form \eqref{semialgebraicbasicform}. More precisely, let
$$
f^*:=\inf\{f(x):x\in X\}.
$$
Then $f^*=\sup\{a\in\rr:f(x)-a>0\hbox{ for }x\in X\}$, and by Putinar's result \cite{Putinar}, 
$$
f^*=\sup\{a\in\rr:f-a \in P(g_1,\ldots,g_r)\},
$$
or equivalently
$$
f^*= \inf\{L(f) : L : \rr[x ] \to \rr \hbox{ is linear, } L(1) = 1,\, L(P(g_1,\ldots,g_r)) \subset [0,\infty)\}.
$$
Denote
$$
P_k(g_1,\ldots,g_r):=\Big\{\sigma_0g_0+\cdots+\sigma_rg_r\in P(g_1,\ldots,g_r):\deg \sigma_ig_i\le k,\,i=0,\ldots,r\Big\},
$$
where we set $g_0=1$. 
Lasserre considered the following optimization problems:
\[
\begin{split}
&\hbox{maximize }a\in\rr: f-a\in P_k(g_1,\ldots,g_r),\\
&\hbox{minimize }L(f)\hbox{ for }L:\rr[x]_k\to\rr,\hbox{ linear, }\, L(1) = 1,\, L(P_k(g_1,\ldots,g_r)) \subset [0,\infty),
\end{split}
\]
where $\rr[x]_k$ is the linear space of polynomials $h\in\rr[x]$ such that $\deg h\le k$. Set
\[
\begin{split}
a_k^*&:=\sup\{a\in\rr:f-a\in P_k(g_1,\ldots,g_r)\},\\
l_k^*&:=\inf\{L(f):L:\rr[x]_k\to\rr\hbox{ is linear, } L(1) = 1,\, L(P_k(g_1,\ldots,g_r)) \subset [0,\infty)\},
\end{split}
\]
\noindent for sufficiently large  $k\in\nn$.
Lasserre proved that $(a_k^*)$, $(l_k^*)$ are increasing sequences that converge to $f^*$ and  $a_k^*\le l_k^*\le f^*$ for $k\in\nn$. 

 We obtain a version of the Lasserre theorem for 
\begin{multline*}
\KS_k(g,g_1,\ldots,g_r):
=\{\sigma_0+\sigma_1 g+\vf(g_1) g_1+\cdots +\vf(g_r)g_r \in \KS(g,g_1,\ldots,g_r):\\
 \deg \sigma_0,\deg \sigma_1 g, 
\deg g_i\vf(g_i)\le k\}, \quad k\in\nn.
\end{multline*}
The implication \eqref{equivalence1} allows us to apply the Lasserre  algorithm of minimizing polynomials on basic compact semialgebraic sets by using $\KS_k(g,g_1,\ldots,g_r)$ instead of $P_k(g,g_1,\ldots,g_r)$ (see Remark \ref{KSoptimization}). 
 Consideration of the cones $\KS_k(g,g_1,\ldots,g_r)$ potentially simplifies the problem of minimizing polynomials on the set $X$, since these cones are properly  contained in $P_k(g,g_1,\ldots,g_r)$.

In Proposition \ref{concave2}, we present another method of minimizing a polynomial $f$ on a compact basic semialgebraic set $X$, say $X\subset \{x\in\rr^n:|x|\le R\}$. Namely, for any $\epsilon>0$, we give an effective procedure for calculating a polynomial $h\in\Phi(g_1,\ldots,g_r)$  
such that 
\begin{equation*}
\forall_{|y|\le R}\;\exists_{x\in X}\;f(y)-h(y)\ge f(x)-h(x)-\epsilon,
\end{equation*}
 and $|h(x)|<\epsilon$ for $x\in X$. In particular,
$$
f^*-2\epsilon \le \inf \{f(y)-h(y):|y|\le R\} \le f^*+2\epsilon.
$$
Thus, the problem of approximate minimization of $f$ can be reduced to the simpler case when the set $X$ is described by one inequality $R^2-|x|^2\ge 0$ (see Remark~\ref{Algorithm2}). In this case M.~Schweighofer \cite{Schweighofer1} gave the rate of convergence of the sequence 
$$
a_k^{**}:=\sup \{a\in\rr:f-h-a\in P_k(R^2-|y|^2)\}\to f^{**},\quad\hbox{as }k\to\infty,
$$
where $f^{**}:= \inf \{f(y)-h(y):|y|\le R\}$.

\subsection{Convexifying positive polynomials.}\label{convexnotation}

We will prove Theorem \ref{convexonconvexsetnew} which, we believe, is of independent interest: for any  polynomial $f$  positive on  a convex closed set  $X$,  whose   leading form is  strictly positive  in $\rr^n\setminus\{0\}$, there exists $N_0\in\nn$ such that for any integer $N\ge N_0$ the polynomial $\varphi_N(x)=(1+|x|^2)^Nf(x)$ is a strictly convex function on~$X$. In the case of homogeneous polynomials and $X=\rr^n$ the same result was obtained by Reznick \cite[Theorem 4.6]{Reznick2}, \cite[Theorem 3.12]{Reznick}.

First in Section  \ref{onevariableconvexsection} we consider the univariate case, and
 we  give an explicit bound for $N_0$ in terms of the coefficients of $f$ and the infimum~$f^*$. 
 We  also give an example to show that  $N_0$ cannot be a function of the degree of $f$ alone.

In Section \ref{convexityatinfinitysection} we prove that the convexity at infinity  of $\varphi_N(x)=(1+|x|^2)^Nf(x)$ for sufficiently large $N$ is equivalent to the strict positivity of the leading form of $f$ 
(Proposition \ref{pconvex1}). Moreover, in Corollary \ref{Reznickconpositivcor} we obtain an interpretation of  Reznick's result \cite[Theorem 3.12]{Reznick}  in terms of convexity.
As a consequence  of  Theorem \ref{convexonconvexsetnew} we  prove in  Corollary \ref{conpositivcor} that, if $X$ is a convex set containing at least two points, 
and  $d>\deg f$ is an even integer, then the following conditions are equivalent:

\begin{enumerate}
\item[(i)]  $f$ is nonnegative on  $X$,
\item[(iii)] for any $a,b>0$ there exists $N_0\in\nn$ such that for any integer $N\ge N_0$   the polynomial $\varphi_N(x)=(1+|x|^2)^N(f(x)+a|x|^{d}+b)$ is a strictly convex function on $X$. 
\end{enumerate}

Finally, we propose the following  algorithm. Given a compact convex  semialgebraic set $X$ and a polynomial $f:\rr^n \to \rr$, assume that  $f$ is  positive on $X$. Then by our convexification result, there exists an integer $N$ such that $\varphi_{N,\xi}(x):=(1+|x-\xi|^2)^Nf(x)$ is a convex function for any $\xi \in X$. (Actually one can take $N=6$.)
Choose  any $a_0 \in X$, and then by induction  set 
$
a_\nu:= \hbox{argmin}_X \,  \varphi_{N,{a_{\nu-1}}}.
$
In Theorem \ref{tw1extra}  we state that the limit  $a^*=\lim_{\nu\to\infty}a_\nu$ exists; moreover,  $a^* $ is a critical point of $f$ on $X$.
The proof requires subtle estimates for the lengths of gradient trajectories of $f$ on $X$.
Since the set of critical {{}values} is finite, this result gives a method for finding the minimum of $f$ on $X$.

\section{Approximation of nonnegative polynomials}\label{sectionapproximation1}

Let $X\subset \rr^n$ be a closed basic semialgebraic set defined by $g_1,\ldots,g_r\in \rr[x]$, i.e. of the form \eqref{semialgebraicbasicform}.

\begin{twr}\label{approximation1}
Let $f\in\rr[x]$ be  nonnegative on the set $X$. Then there exists a sequence $f_\nu\in P(g_1,\ldots,g_r)$, $\nu\in\nn$, that is uniformly convergent to $f$ on  compact subsets. Moreover,  $f_\nu$ can be chosen from the cone ${\KSS}(g_1,\ldots,g_r)$ \footnote{Recall that
${\KSS}(g_1,\ldots,g_r)=\{\sigma_0+\vf(g_1)g_1+\cdots+\vf(g_r)g_r:\sigma_0\in\sum\rr[x]^2,\, \vf\in\sum\rr[t]^2\}$.}.
 In particular $f_\nu$ converges to $f$ in the $l_1$-norm. 
\end{twr}

\begin{proof} Take any positive constants $\ve,\delta,A,B$. 
By the Weierstrass Approximation Theorem 
there exists a polynomial $\vf_{\ve,\delta,A,B}\in\rr[t]$ 
such that
\begin{alignat}{2}\label{stoneWeierstrass2}
\vf_{\ve,\delta,A,B}(t)&>B  &\quad  &\hbox{for } t\in [-A,-\delta],
\\
\label{stoneWeierstrass3}
\vf_{\ve,\delta,A,B}(t)&<\ve  &      &\hbox{for } t\in[0,A].
\end{alignat}
Taking $\vf^2_{\ve,\delta,A,B}$ if necessary, we may additionally assume that
\begin{equation}\label{stoneWeierstrass1}
\vf_{\ve,\delta,A,B}(t)\ge 0\quad\hbox{for } t\in\rr.
\end{equation}
Set
$$
g_{i,\ve,\delta,A,B}:=g_i\cdot \vf_{\ve,\delta,A,B}\circ g_i\quad\hbox{for } i=1,\ldots,r.
$$
Every nonnegative univariate polynomial is a sum of squares of polynomials, hence by \eqref{stoneWeierstrass1} we have 
\begin{equation}\label{gisos1}
\vf_{\ve,\delta,A,B}\circ g_i\in\sum \rr[x]^2\quad\hbox{for } i=1,\ldots,r.
\end{equation}

Since the sequence $h_\nu=f+\frac{1}{\nu}$, $\nu\in \nn$, uniformly converges to $f$, we may assume that $f$ is positive on $X$. 

Fix an arbitrary $R>1$ and let $M>1$ be a constant such that
\begin{equation}\label{esimationf1}
f(x)\ge -M \quad \hbox{for } x\in[-R,R]^n.
\end{equation}

Since $f$ is positive on $X$, we have
$$
X\cap [-R,R]^n\subset G_1, 
$$
where the set $G_1:= \{x\in[-R,R]^n:f(x)>0\}$ is open in $[-R,R]^n$. As $X\cap [-R,R]^n$ is a compact set, there exists $\eta>0$ such that
$$
G_2:=\{x\in[-R,R]^n:\dist(x,X)\le \eta\}\subset G_1.
$$
Since $\overline{[-R,R]^n\setminus G_2}=\{x\in[-R,R]^n:\dist (x,X)\ge \eta\}$ is also compact,  by the definition of $X$ there exists $\delta\in (0,1]$ such that 
\begin{equation}\label{estimationg1}
G_3:=\{x\in[-R,R]^n: g_i(x)\ge-\delta \hbox{ for } i=1,\ldots,r\}\subset G_2.
\end{equation}

Let 
$$
f^*_R:=\min\{f(x):x\in G_2\}.
$$
Obviously $f^*_R>0$.

Let $A\ge 1$ be a constant such that
$$
|g_i(x)|\le A\quad\hbox{for } x\in [-R,R]^n,\, i=1,\ldots,r.
$$
Take
\begin{equation}\label{epsilonB}
 \ve:=\frac{f^*_R}{(r+1)A},\qquad B:=A\frac{M+r\ve  }{\delta}.
\end{equation}
 
\begin{lemat}\label{interpolation1}
For any $x\in[-R,R]^n$ we have  $f(x)-\sum_{i=1}^r  g_{i,\ve,\delta,A,B}(x)>0$.
\end{lemat}

\begin{proof} Take $x\in [-R,R]^n$. 

If $x\in X$, then $g_i(x)\ge 0$ for $i=1,\ldots,r$, and by \eqref{stoneWeierstrass3},
$$
 g_{i,\ve,\delta,A,B}(x)= g_i(x)\cdot\vf_{\ve,\delta,A,B}\circ (g_i(x))\le A \ve <\frac{f^*_R}{r} \quad\hbox{for } i=1,\ldots,r.
$$
So  
$$
f(x)-\sum_{i=1}^r  g_{i,\ve,\delta,A,B}(x)>f^*_R-r\frac{f^*_R}{r}\ge 0,
$$
and the assertion holds.

Let $x\in G_3\setminus X$. Without loss of generality we may assume that 
$$
g_1(x),\ldots,g_k(x)\ge 0\quad\hbox{and}\quad g_{k+1}(x),\ldots,g_r(x)<0
$$
for some $0\le k<r$. Then by  \eqref{stoneWeierstrass3},
$$
 g_{i,\ve,\delta,A,B}(x) \le A \ve <\frac{f^*_R}{r} \quad\hbox{for } i=1,\ldots,k,
$$
and by \eqref{stoneWeierstrass1},
$$
 g_{i,\ve,\delta,A,B}(x)<0\quad\hbox{for } i=k+1,\ldots,r.
$$
Consequently, $f(x)-\sum_{i=1}^r  g_{i,\ve,\delta,A,B}(x)>f^*_R-k\frac{f^*_R}{r} > 0$, and the assertion holds.

Let now $x\in [-R,R]^n\setminus G_3$. Without loss of generality we may assume that 
$$
g_1(x),\ldots,g_k(x)\ge 0, \quad 0>g_{k+1}(x),\ldots,g_l(x) \ge -\delta,\quad  g_{l+1}(x),\ldots,g_r(x) < -\delta,
$$
where $0\le k\le l<r$. Then
$$
g_{i,\ve,\delta,A,B}(x) <\frac{f^*_R}{r+1}\quad\hbox{for } i=1,\ldots,k,
$$ 
and 
$$
g_{i,\ve,\delta,A,B}(x)<0\quad\hbox{for } i=k+1,\ldots,l.
$$
By \eqref{stoneWeierstrass2} we see that
$$
 g_{i,\ve,\delta,A,B}(x)<A(-M - r\ve)\le-M-\frac{r f^*_R}{r+1} \quad\hbox{for } i=l+1,\ldots,r.
$$
Summing up,
$$
f(x)-\sum_{i=1}^r  g_{i,\ve,\delta,A,B}(x)>-M-k\frac{f^*_R}{r+1} +(r-l)\left(M+\frac{r f^*_R}{r+1}\right) > 0,
$$
as desired.
\end{proof}

\begin{remark}\label{remark1}
The polynomial $\vf_{\ve,\delta,A,B}(t)$ in the above proof can be chosen  of the form
$$
\vf(t)=\left(\frac{1}{A}t-1+\frac{\delta}{2A}\right)^{2N}
$$
with $N\log(1-\frac{\delta}{2A})^2<\log \ve$, $N\log (\frac{\delta}{2A})^2<\log \ve$ and $N\log(1+\frac{\delta}{2A})^2>\log B$.  M.~Schweighofer \cite[Lemma 2.3]{Schweighofer2} in a similar problem proposes a polynomial $\varphi$ of the form $\varphi(t)=as(at-1)^{2N}$ for some $s\in\nn$ and $a>0$.
\end{remark}

By Lemma \ref{interpolation1}, for any $R>0$ there exists $\vf_{R}\in\sum\rr[t]^2$ such that
$$
f(x)-\sum_{i=1}^r\vf_{R} ( g_i(x))g_i(x)>0\quad\hbox{for } x\in[-R,R]^n.
$$
By Fact \ref{fact1}  in the Introduction,  it is easy to see that 
$f(x)-\sum_{i=1}^r\vf_{R} ( g_i(x))g_i(x)$ can be approximated in the $l_1$-norm by sums of squares of polynomials and it 
 can be approximated uniformly on $[-R,R]^n$ by sums of squares of polynomials.  Consequently, $f$ can be approximated uniformly on $[-R,R]^n$  (in particular in the $l_1$-norm) by polynomials from the cone $\KSS(g_1,\ldots,g_r)$.  Hence we deduce the assertion of Theorem \ref{approximation1}.
\end{proof}

\section{Quantitative aspects of Theorem \ref{approximation1}}\label{quantitative}
 In order to  estimate 
the rate of convergence in Lasserre's relaxation method \cite{Las-1} we show 
how to bound the degree of the polynomial $\varphi$ in  Theorem \ref{approximation1}. The key point is to find a lower bound for $\delta$ which satisfies the inclusion \eqref{estimationg1}.

 Assume now that $X$ is a compact  set of the  form
\begin{equation}\label{semialgebraicbasicformcompact}
X=\{x\in \rr^n:g_1(x) \ge 0,\ldots,g_r(x)\ge 0\},
\end{equation}
where $g_1,\ldots,g_r\in \rr[x]$. Choose  $R>0$ large enough so that   $g_0(x)=R^2-|x|^2$ is nonnegative  polynomial on $X$.
We now define a cone
\begin{equation*}
\KS(g_0,\ldots,g_r):=\Big\{\sigma_0+\sigma_1g_0+\sum_{i=1}^r\vf(g_i)g_i: \\ 
\sigma_0,\sigma_1\in\sum\rr[x]^2,\, \vf\in\sum\rr[t]^2\Big\}.
\end{equation*}

By the argument in the proof of Theorem \ref{approximation1} we obtain 

\begin{cor}\label{approximationPutinar}
If $f\in\rr[x]$ is strictly positive on the set $X$, then $f\in \KS(g_0,\ldots,g_r)$.
\end{cor}

\begin{proof} By Lemma \ref{interpolation1}, there exists $\varphi\in\sum\rr[t]^2$ such that 
$$h(x)=f(x)-\sum_{i=1}^r \varphi(g_i(x))g_i(x)>0$$
 for $|x|\le R$. Since $\{x\in\rr^n:|x|\le R\} =\{x\in \rr^n:g_0(x)\ge0\}$, 
  Putinar's Positivstellensatz  
  (or  Schm\"udgen's Positivstellensatz, 
   because $P(g_0)=T(g_0)$)  yields
$$
h\in\left\{\sigma_0+\sigma_1g_0:\sigma_0,\sigma_1\in\sum\rr[x]^2\right\} = \KS(g_0),
$$
which completes the proof.
\end{proof}

Corollary \ref{approximationPutinar} also   follows from Schweighofer's result \cite[Lemma 2.3]{Schweighofer2} and the Putinar theorem.


\begin{remark}\label{KSoptimization}
We may use the Lasserre algorithm for minimization of a polynomial $f$ on a compact basic semialgebraic set $X$ by using ${\KS}(g,g_1,\ldots,g_r)$ instead of $P(g,g_1,\ldots,g_r)$. In fact, we can use the set ${\KS}_k(g,g_1,\ldots,g_r)$ consisting of all $\sigma_0+\sigma_1g+\sum_{i=1}^r\vf(g_i)g_i\in {\KS}(g,g_1,\ldots,g_r)$ such that $\deg \sigma_0\le k$, $\deg \sigma_1g \le k$ and $\deg \varphi(g_i)g_i\le k$ for $i=1,\ldots,r$. Consider the following optimization problems:

\begin{itemize}
\item  maximize  $a\in\rr$ such that $f-a\in {\KS}_k(g,g_1,\ldots,g_r)$,
\item minimize  $L(f)$ for $L:\rr[x]_k\to\rr$, linear, $L(1) = 1$, 
$ L({\KS}_k({{}g,}g_1,\ldots,g_r)) \subset [0,\infty)$.
\end{itemize}

\noindent
Denote
\[
\begin{split}
u_k^*&:=\sup\{a\in\rr:f-a\in {\KS}_k({{}g,}g_1,\ldots,g_r)\},\\
v_k^*&:=\inf\{L(f):L:\rr[x]_k\to\rr\hbox{ is linear, }\, L(1) = 1,\, L({\KS}_k({{}g,}g_1,\ldots,g_r)) \subset [0,\infty)\},
\end{split}
\]
for sufficiently large  $k\in\nn$.
We see that $(u_k^*)$, $(v_k^*)$ are increasing sequences that converge to $f^*$ (by Corollary \ref{approximationPutinar}) and  $u_k^*\le v_k^*\le f^*$ for $k\in\nn$.\hfill$\square$
\end{remark}

\subsection{Quantitative {\L}ojasiewicz inequality}\label{Lquantitative}
Let $g_1,\ldots,g_r\in \rr[x]$, and let $G:\rr^n\to\rr$ be  defined by
\begin{equation}\label{eqdefh}
G(x)=\max\{0,-g_1(x),\ldots,-g_r(x)\},\quad x\in\rr^n.
\end{equation}
Then 
$
X=\{x\in \rr^n:g_1(x) \ge 0,\ldots,g_r(x)\ge 0\}=G^{-1}(0).
$
Moreover, 
$$
\graph G=Y_0\cup Y_1\cup\cdots\cup Y_r,
$$
where
$
Y_0=X\times\{0\},
$
$$
Y_i=\{(x,y)\in\rr^n\times \rr:y=-g_i(x),\,g_i(x)\le 0,\,g_i(x)\le g_j(x)\;\hbox{for }j\ne i\},
$$
for $i=1,\ldots,r$. Note that each  set $Y_i$, $i=0,\ldots,r$, is defined by $r$ inequalities and one equation. Let $d=\max\{\deg g_1,\ldots,\deg g_r\}$. We now state the well-known \emph{{\L}ojasiewicz inequality} in a quantitative version proved in \cite[Corollary 2.3]{KSS} (see also \cite[Corollary 10]{KS}): there exist $C,\mathcal{L}>0$ such that 
\begin{equation}\label{lojin}
G(x)\ge C\left(\frac{\dist(x,X)}{1+|x|^d}\right)^{\mathcal{L}},\quad x\in \rr^n,
\end{equation}
with
\begin{equation}\label{estimateh}
\mathcal{L}\le d(6d-3)^{n+r-1}.
\end{equation}

It follows from \eqref{lojin} that  for every 
$\rho>0$  there exists $C_{\rho}>0$ such that
\begin{equation}\label{eqKS}
G(x)\ge C_{\rho} \dist(x,X)^{\mathcal{L}} \quad \text{for any }x\in B(\rho),
\end{equation}
where $B(\rho)=\{x\in\rr^n:|x|\le \rho\}$. 
Fix $R>0$ such that  $X\subset B(R)$. Assume that \eqref{eqKS} holds with  fixed $C^\prime=C_R$ and $\mathcal{L}$.

\begin{fact}\label{fact2}
Let $\eta>0$. Set $\delta_0=C'\eta^{\mathcal{L}}$. Then for any $0<\delta\le \delta_0$, 
$$
\{x\in B(R): g_i(x)\ge-\delta \hbox{ for } i={{}1},\ldots,r\}\subset \{x\in B(R):\dist(x,X)\le \eta\}.
$$
\end{fact}

Indeed, take $x\in B(R) \setminus X$ such that $g_i(x)\ge-\delta$ for $i=0,\ldots,r$. Let $G$ be the function defined by \eqref{eqdefh}. Hence by \eqref{eqKS}, 
$$
\delta\ge \max\{-g_1(x),\ldots,-g_r(x)\}= G(x)\ge C'\dist (x,X)^{\mathcal{L}}.
$$  
Thus for $0<\delta\le \delta_0$ we deduce the assertion of Fact \ref{fact2}.\hfill$\square$
  
\subsection{Approximation}

For $\nu=(\nu_1,\ldots,\nu_n)\in \nn^n$ we set $|\nu|=\nu_1+\cdots+\nu_n$ and $a^\nu=a_1^{\nu_1}\cdots a_n^{\nu_n}$, where $a=(a_1,\ldots,a_n)\in\rr^n$. For  $h\in\rr[x]$ of the form
\begin{equation*}\label{formoffmultidimensional}
h(x)= \sum_{j=0}^d\sum_{|\nu| =j}a_{\nu}x^{\nu},
\end{equation*}
we define
\begin{equation*}
\mathbb{A}(h,R)=\sum_{j={{}0} 
}^d \sum_{|\nu| =j}|a_{\nu}| R^{j},\qquad
\mathbb{B}(h,R)=
\sum_{j=1}^d \sum_{|\nu| =j}j|a_{\nu}| R^{j-1}\quad \hbox{for }R> 0.
\end{equation*}

Then for $x\in B(R)$ we have $|h(x)|\le \mathbb{A}({{}h},R)$ and by the Euler formula for  homogeneous functions, $|\grad h(x)|\le \mathbb{B}({{}h},R)$.

Using a similar argument to the one for Theorem \ref{approximation1} we obtain the following

\begin{prop}\label{concave2}
Let $f\in\rr[x]$, let $X$ be a semialgebraic set of the form \eqref{semialgebraicbasicform} such that $X\subset B(R)$, $R>0$, and let $g_1,\ldots,g_r\in\rr[x]$ be polynomials satisfying \eqref{eqKS} with fixed $C,\mathcal{L}>0$. Take $M,A\in \rr$ such that
$$
M\ge \max\{1,\mathbb{A}(f,R),\mathbb{B}(f,R)\},\quad
A\ge \max\{1,\mathbb{A}(g_i,R)\}
\quad\hbox{for }i=1,\ldots,r.
$$
Take $\epsilon>0$, and set
$$
\vf(t)=\left(\frac{1}{A}t-1+\frac{\delta}{2A}\right)^{2N},
$$
where 
$$
0<\delta\le \min\left\{A,C\left(\frac{\epsilon}{2M}\right)^{\mathcal{L}}\right\},\ N\ge \max\left\{\frac{(r-1)A-1}{2},\frac{A(2M+1-\delta)}{\delta^2},
\frac{2rA-\epsilon}{2\epsilon}\right\}.
$$ 
Then the function 
$$
h(x)=\sum_{i={{}1}}^r \varphi(g_i(x))g_i(x)\in \Phi(g_1,\ldots,g_r)
$$ 
satisfies the following conditions:
\begin{gather}\label{eqconcave2}
0\le h(x)<\epsilon\quad\hbox{for }x\in X,
\\
\label{eqconcave3}
\forall_{|y|\le R}\ \exists_{x\in X}\ f(y)-h(y)\ge f(x)-h(x)-\epsilon.
\end{gather}
\end{prop}

\begin{proof}
It is easy to see that for the function 
$$
\phi(t)=t\left(\frac{1}{A}t-1+\frac{\delta}{2A}\right)^{2N},
$$
where $0<\delta<A$, $N>\frac{(r-1)A-1}{2}$, 
$
N>\frac{A(2M+1-\delta)}{\delta^2},
$ 
we have
\begin{alignat}{2}\label{psiepsilon1}
\phi(t)&<\frac{A}{2N+1}& &\hbox{for }t\in[0,A],
\\
\label{psiepsilon2}
\phi(t)&\le -2M-\frac{(r-1)A}{2N+1}&\quad&\hbox{for }t\le -\delta.
\end{alignat}

From the assumptions of $M$ and $A$ we have $|f(x)|\le M$, $|\grad f(x)|\le M$ and $|g_i(x)|\le A$ for  
 $i=1,\ldots,r$ and $x\in\rr^n$ such that $|x|\le R$.

Take any $\epsilon>0$. Let 
\begin{align*}
Y &:=\{y\in\rr^n:|y|\le R\,\land\,\exists_{x\in X}\ f(y)\ge f(x)-\epsilon/2\},\\
\eta &:=\frac{\epsilon}{2M},\\
Y_1 &:=\{y\in \rr^n:|y|\le R\,\land\,\dist (y,X)\le \eta\}.
\end{align*}
By the Mean Value Theorem, $Y_1\subset Y$. From Fact \ref{fact2}, for $0<\delta\le C\eta^{\mathcal{L}}$ we have 
$$
Y_2:=\{x\in \rr^n:|x|\le R\,\land\, g_i(x)\ge-\delta \hbox{ for } i=0,\ldots,r\}\subset Y_1\subset Y. 
$$

Obviously $h(x)\ge 0$ for $x\in X$.
Since $h(x)=\sum_{i=1}^r\phi(g_i(x))$ and $g_i(x)\in[0,A]$ for $x\in X$, by \eqref{psiepsilon1} and the assumption $N\ge \frac{2rA-\epsilon}{2\epsilon}\ge\frac{rA-\epsilon}{2\epsilon} $ we obtain \eqref{eqconcave2}. 

Now we prove \eqref{eqconcave3}. Obviously it holds for $y\in X$.

Take $y\in Y_2\setminus X$. Without loss of generality we may assume that 
$$
g_1(y),\ldots,g_k(y)\ge 0\quad\hbox{and}\quad g_{k+1}(y),\ldots,g_r(y)<0
$$
for some $0\le k<r$. Then there exists $x\in X$ such that $f(y)\ge f(x)-\frac{\epsilon}{2}$. So, \eqref{psiepsilon1} and the assumption  $N\ge \frac{2rA-\epsilon}{2\epsilon}$ give
$$
f(y)-h(y)\ge f(x)-\frac{\epsilon}{2}-h(y)\ge f(x)-\frac{\epsilon}{2}-\sum_{i=1}^k\phi(g_i(y))\ge f(x)-\epsilon\ge f(x)-h(x)-\epsilon.
$$
This proves \eqref{eqconcave3} for $y\in Y_2\setminus X$. 

Let now $y\in \{x\in \rr^n:|x|\le R,\, x\not\in Y_2\}$. Without loss of generality we may assume that 
$$
g_1(y),\ldots,g_k(y)\ge 0, \quad 0>g_{k+1}(y),\ldots,g_l(y) \ge -\delta,\quad  g_{l+1}(y),\ldots,g_r(y) < -\delta,
$$
where $0\le k\le l<r$. Then, by the choice of $M$, the assumption 
$N\ge \frac{A(2M+1-\delta)}{\delta^2}$ and \eqref{psiepsilon2} we see that $h(y)\le -2M$, and so for any $x\in X$ we have
$$
f(y)-h(y)\ge -M+2M\ge f(x)\ge f(x)-h(x)\ge f(x)-h(x)-\epsilon.
$$
This gives \eqref{eqconcave3} in the  case under consideration and ends the proof. 
\end{proof}

\begin{remark}\label{remmuconcave}
If we assume that $g_1,\ldots,g_r$ are $\mu$-strongly concave polynomials, i.e., 
\begin{equation*}\label{eqconcave1}
g_i(y)\le g_i(x)+\langle y-x,\grad g_i(x)\rangle -\frac{\mu}{2}|y-x|^2 \quad\hbox{for }x,y\in\rr^n,
\end{equation*}
where $\mu>0$ and $\langle\cdot\,,\cdot\rangle$ is the standard scalar product, then the assertion of Fact~\ref{fact2} holds with $\delta_0=\eta^2\mu/2$. Hence, Proposition \ref{concave2} holds with $0<\delta\le \min\big\{A,\frac{\epsilon^2\mu}{8M^2}\big\}$.
\end{remark}

\begin{remark}\label{Algorithm2}
We can use Proposition \ref{concave2} to minimize a polynomial $f$ on a compact basic semialgebraic set $X$. Let $X\subset \{x\in\rr^n:|x|\le R\}$. Then for any $\epsilon>0$, we can effectively compute a polynomial $h(x)=\sum_{i=i}^r \varphi(g_i(x))g_i(x)$, where  $\vf\in\sum\rr[t]^2$, such that 
$$
f^*-2\epsilon \le \inf \{f(y)-h(y):|y|\le R\} \le f^*+2\epsilon.
$$
To approximate $f^*$, we can minimize $f-h$ on $B(R)$. To this end we may compute 
$$
a_k^{**}:=\sup \{a\in\rr:f-h-a\in P_k(R^2-|y|^2)\}\quad\hbox{for }k\in\nn.
$$
By the Putinar Theorem (or the Schm\"udgen Theorem) we see that
$$
a_k^{**}\to f^{**}\quad\hbox{as }k\to\infty,
$$
where $f^{**}:= \inf \{f(y)-h(y):|y|\le R\}$.

Minimization of $f-h$ on $B(R)$ is much simpler than minimizing  $f$ on $X$, because the set $B(R)$ is described by one inequality $R^2-|x|^2\ge 0$. In this case M.~Schweighofer \cite{Schweighofer1} gave the rate of convergence of the sequence $a_k^{**}$: 
$$
f^{**}-a_k^{**}\le\frac{c}{\sqrt[d]{k}}
$$
for some constant $c\in\nn$ depending on $f$ and $R^2-|y|^2$ and some constant $d\in\nn$ depending on $R^2-|y|^2$.
\end{remark}

\section{Convex polynomials in one variable}\label{onevariableconvexsection}

 We denote  by  $ \nn^*$ the set of strictly positive integers.  In this section $x$ denotes a single variable. Let $f\in\rr[x]$ be a nonzero polynomial. For any $N\in \nn^*$  we define the following polynomial:
\begin{equation}\label{vfpolynomialone10}
\varphi_N(x):=(1+x^2)^Nf(x).
\end{equation}
We will find $N_0\in\nn^*$ such that for $N\ge N_0$ the polynomial $\vf_N$ is strongly convex on a closed interval $I\subset \rr$, provided $f$  is positive on $I$.

For positive numbers $m,R,D$ we set
\begin{equation}\label{formulanaN}
\mathcal{N}(m,R,D): = \max\left\{ \frac{D}{m}+\frac{m}{16D},\, \frac{(1+R^2)D}{Rm}+1,\frac{4D^2}{m^2}+2,\, \frac{(1+R^2)D}{2m}  \right\}.
\end{equation}

We first  prove that if  $f$ is  a 
$C^2$ function positive on a bounded interval $I$, then   $\varphi_N(x)=(1+x^2)^Nf(x) $ is convex for every  $N$ sufficiently large.  We formulate this lemma for $C^2$ functions because restricting to polynomials does not  simplify the proof considerably.

\begin{lemat}\label{convexonevariableoncompact}
Let   $f$ be a $C^2$ function 
positive on an interval  
  $I=[a,b]\subset \rr$,  and let $R\ge \max\{|a|,|b|\}$. If $m,D>0$ satisfy the conditions  
\begin{equation}\label{estimatef0}
m\le \min\{f(x):x\in I\},
\end{equation}
\begin{equation}\label{eqD10}
|f'(x)|\le D ,\quad |f''(x)|\le D \quad \hbox{for } |x|\le R,
\end{equation}
then for any 
 $N\in\nn$ satisfying
\begin{equation}\label{uniformestimates1}
N> \mathcal{N}(m,R,D)
\end{equation}
we have $\vf_N''(x)>0$ for $x\in I$, thus $\vf_N(x)$ is strongly
 convex on $I$.
\end{lemat}

\begin{proof} 
Denote $P_N=A_N+B_N+Q_N+T_N$, where
\begin{align*}
A_N(x)&=4N(N-1)x^2f(x),\qquad 
B_N(x)=2N(1+x^2)f(x),\\
Q_N(x)&=4N(1+x^2)xf'(x),\qquad
T_N(x)=(1+x^2)^2f''(x).
\end{align*}
Then 
\begin{equation}\label{eqvfN}
\vf_N''(x)=(1+x^2)^{N-2}P_N(x).
\end{equation}

Let 
$N\in \nn$ satisfy \eqref{uniformestimates1}. To  prove that $\vf_N$ is convex on $I$ 
 we will proceed in several steps.
From \eqref{estimatef0} and \eqref{eqD10} we obtain
\begin{alignat}{2}\label{eq30}
A_N(x)& \ge 4N(N-1)x^2m & &\hbox{for } x\in I,
\\
\label{eq40}
B_N(x) &\ge 2N(1+x^2)m &  &\hbox{for } x\in I,
\\
\label{eq50}
Q_N(x) &\ge -4N(1+x^2)|x|D &\quad &\hbox{for } |x|\le R,
\\
\label{eq60}
T_N(x) &\ge -(1+x^2)^2D & &\hbox{for } |x|\le R.
\end{alignat}

Since $N$ satisfy \eqref{uniformestimates1}, we have 
\begin{equation}\label{estimateN0}
N\ge \frac{D}{m}+\frac{m}{16D}.
\end{equation}
Note that then
\begin{equation}\label{estimation10}
\frac{m}{4D}\le\sqrt{\frac{Nm-D}{D}} .
\end{equation}
Assume now  that $x\in I$, $|x|<\frac{m}{4D}$. Then obviously $A_N(x)\ge 0$. By \eqref{eq40} and \eqref{eq50} 
we have
\begin{equation*}\label{eq80}
\frac{1}{2}B_N(x)+Q_N(x)>0.
\end{equation*}
Also  by \eqref{eq40},  \eqref{eq60} and \eqref {estimation10}, 
\begin{equation*}\label{eq70}
\frac{1}{2}B_N(x)+T_N(x)>0.
\end{equation*}
So for $N$ satisfying \eqref{estimateN0} we have $P_N(x)>0$, and consequently by \eqref{eqvfN},
\begin{equation}\label{eq90}
\vf_N''(x)>0\quad\hbox{for } x\in I,\,|x|< \frac{m}{4D}.
\end{equation}

We have to show now that $P_N(x)>0$ for $x\in I$, $\frac{m}{4D}\le |x|\le R$. 
By \eqref{uniformestimates1} we have
\begin{equation}\label{eqN0}
N >\max\left\{\frac{(1+R^2)D}{Rm}+1,\frac{4D^2}{m^2}+2\right\}.
\end{equation}
By \eqref{eq30} and \eqref{eq50} we see that
\begin{equation}\label{estimation911}
A_N(x)+Q_N(x)\ge (-D|x|^2+(N-1)m|x|-D)4N|x| \quad\hbox{for } x\in I,\,|x|\le R,
\end{equation}
and by \eqref{eqN0},
$$
-D\left(\frac{m}{4D}\right)^2+(N-1)m \frac{m}{4D}-D>0
$$
and
$$
-DR^2+(N-1)mR-D>0.
$$
Hence $-D|x|^2+(N-1)m|x|-D>0$ for $\frac{m}{4D}\le |x|\le R$, and \eqref{estimation911} gives
\begin{equation}\label{eq1000}
A_N(x)+Q_N(x)>0\quad \hbox{for } x\in I,\,\frac{m}{4D}\le |x|\le R.
\end{equation}
By \eqref{uniformestimates1} 
we have
\begin{equation*}\label{estimateN20}
N>\frac{(1+R^2)D}{2m};
\end{equation*}
then, by \eqref{eq40} and \eqref{eq60}, we obtain
\begin{equation}\label{eq110}
B_N(x)+T_N(x)>0\quad \hbox{for } x\in I,\,\frac{m}{4D}\le |x|\le R.
\end{equation}
Consequently, by \eqref{eq1000}, \eqref{eq110} and \eqref{eqvfN},
 we have  
\begin{equation}\label{eq901}
\vf_N''(x)>0\quad\hbox{for } x\in I,\,\frac{m}{4D}\le |x|\le R.
\end{equation}

Summing up, for $N$ satisfying \eqref{uniformestimates1}, by \eqref{eq90} and \eqref{eq901}, 
we have $\vf_N''(x)>0$, $x\in I$, which means that $\vf_N$ is strongly convex on $I$ and Lemma \ref{convexonevariableoncompact} is proved.\end{proof}

\begin{remark}\label{convexonevariableoncompactextra}
Lemma \ref{convexonevariableoncompact} was proved under the assumption that the function $f$ is~$C^2$.
If we assume that $f$ is a polynomial which is positive except possibly at $0\in\rr$, then an analogous argument leads to a strictly convex function $\vf_N$. More precisely, 
let $f\in\rr[x]$ be a polynomial positive on   
  $I=[a,b]$ except possibly at $0\in\rr$, where $0\in (a,b)$.
  Then there exists $N_0\in\nn$ such that for any $N\in\nn$ with $N\ge N_0$ the polynomial  $\vf_N(x)$ is  strictly 
 convex on $I$.
\end{remark}


For a  polynomial  of degree $d$  of the form
\begin{equation}\label{formf0}
f=\sum_{i=0}^d a_ix^{d-i},\quad a_0,\ldots,a_d\in\rr,\quad a_0\ne 0,
\end{equation}
and  $R>0$, we set
\begin{equation*}\label{estimationN20}
D(f,R):=\max\left\{1,\,\sum_{i=0}^{d-1}(d-i)|a_i|R^{d-i-1},\, \sum_{i=0}^{d-2}(d-i)(d-i-1)|a_i|R^{d-i-2}\right\}.
\end{equation*}
We easily see that  for any $D\ge D(f,R)$ the assumption \eqref{eqD10} of Lemma \ref{convexonevariableoncompact} holds. 
If $d>0$, we define
\begin{equation*}\label{estimateN10}
K(f)=1+2\max_{1\le i\le d}\left|\frac{a_i}{a_0}\right|^{1/i}.
\end{equation*}
Obviously $K(f)>0$. It is known that if $f(z)=0$, $z\in \cc$ then $|z| < K(f)$.  
Since for $d\ge 2$ the complex zeroes of $f'$ and $f''$ lie in the convex hull of the set of complex zeroes of $f$, 
\begin{equation}\label{sentenceonf}
\hbox{$f$, $f'$ and $f''$ have no zeroes $x\in\rr$ such that $|x| \ge K(f)$.} 
\end{equation}

We prove a version of  Lemma \ref{convexonevariableoncompact} for a polynomial on an arbitrary interval. (A~version of this lemma, without  explicit  bound for $N$, 
has been proven in the M.Sc. thesis of I.~Fau \cite{Fau}.)

\begin{lemat}\label{convexonevariableoncompact1}\label{convexonevariable}\label{maincorollary1}
Let $f\in\rr[x]$ be positive on a closed interval $I\subset \rr$. Let $m>0$ satisfy \eqref{estimatef0}, 
 and 
let $R\ge K(f)$ and $D\ge D(f,R)$ $($or let $D$ satisfy \eqref{eqD10}$)$. Then for any integer $N>\mathcal{N}(m,R,D)$
 the polynomial $\vf_N(x)$ is strongly
 convex on $I$.
\end{lemat}

\begin{proof} By the same argument as for \eqref{sentenceonf}, we deduce that  $\vf''_N(x)$ for $x\le -R$ 
has the same sign as $\vf''_N(-R)$. Analogously,  $\vf''_N(R)$ and  $\vf''_N(x)$ for $x\ge R$ have the same sign.
 Moreover, $\vf''_N(-R)\ne 0$ and $\vf''_N(R)\ne 0$. So considering the sign of $\vf''_N$ on the intervals  $J_1=I\cap [-R,R]$, $J_2=I\cap [R,+\infty)$ and $J_3=I\cap (-\infty,-R]$
  we deduce the assertion by Lemma \ref{convexonevariableoncompact}.  Note that the strong convexity of $\varphi_N$ is due to the fact  that $f$ is a polynomial. \end{proof}

 \begin{remark}\label{thesameassertion} 
Under the assumptions of Lemma \ref{convexonevariableoncompact}, and with the same argument, we obtain the assertion of this lemma for the function $\varphi_{N,\xi}(x)=(1+(x-\xi)^2)^Nf(x)$ instead of $\varphi_N$, where $\xi\in [-R,R]$, with the bound $N>\mathcal{N}(m,2R,D)$. Hence, the assertion of Lemma \ref{convexonevariableoncompact1} holds for the function $\varphi_{N,\xi}$ with the bound  $N>\mathcal{N}(m,2R,D)$.
\end{remark}

The exponent $N$ in Lemma \ref{convexonevariable} 
actually depends on the coefficients of $f$ even when the degree of $f$ is fixed. 

\begin{exa}\label{exa0} Let $f_k(x)=(x-k)^2+1$. Obviously $f_k$ is a convex function. We have $\vf_N(x)=((x-k)^2+1)(1+x^2)^N$ and $\vf_N(k)=(1+k^2)^N$, $\vf(0)=k^2+1$, $\vf_N(\frac{k}{2})=(\frac{k^2}{4}+1)^{N+1}$. Assume that $\vf_N$ is convex. Then
$$
\left(\frac{k^2}{4}+1\right)^{N+1}\le \frac{1}{2}(k^2+1)+\frac{1}{2}(k^2+1)^N.
$$
So the number $N$ (in Lemma \ref{convexonevariable}) such that 
 the function $\vf_N$ is convex tends to infinity as $k\to \infty$. 
\end{exa}

\begin{remark}\label{arbitraryg}
\rm
By a similar argument to that for Lemmas \ref{convexonevariableoncompact} and  \ref{convexonevariable} one can prove (see \cite{Fau}): for any $f\in\rr[x]$ positive on $\rr$ and any  $g\in\rr[x]$ such that $g(x)>0$ and $g''(x)>0$ for $x\in\rr$ there exists $N_0\in\mathbb{N}$ such that for any $N\ge N_0$ the polynomial $fg^N$ is strictly convex on $\rr$.
\end{remark}

\section{Convexifying  polynomials on compact sets}

Let $x=(x_1,\ldots,x_n)$ be a system of variables and let $f\in\rr[x]$ be a polynomial of the form
\begin{equation}\label{formoffmultidimensional1}
f= \sum_{j=0}^d\sum_{|\nu| =j}a_{\nu}x^{\nu}.
\end{equation}
For $R\ge 0$ define $ \mathbb{D}(f,R):=$
\begin{multline*}
\max\Big\{1,\sqrt{1+R^2}\sum_{j=1}^d \sum_{|\nu| =j}j|a_{\nu}| R^{j-1},(1+R^2)\sum_{j=2}^d \sum_{|\nu| =j} j(j-1) |a_{\nu}|R^{j-2}\Big\}.
\end{multline*}


This will be a bound for the first and the second derivatives in \eqref{estimatefprimefbis1} below.

\begin{twr}\label{convexcompactnew}
Let $f\in\rr[x]$ be  positive on a compact convex set $X\subset \rr^n$ containing at least two points. Set $R=\max\{|x|:x\in X\}$, and let 
\begin{equation}\label{estfcompact}
0<m\le \min\{f(x):x\in X\}.
\end{equation}
Then for any $D\ge \mathbb{D}(f,R)$ and  any integer $N\ge \mathcal{N}(m,R,D)$ the polynomial $\vf_N(x)=(1+x_1^2+\cdots+x_n^2)^Nf(x)$ is strongly 
 convex in $X$.
\end{twr}

\begin{proof} Let 
$$
\mathcal{A}=\{(\alpha,\beta)\in\rr^n\times \rr^n:\langle \alpha,\beta\rangle=0,\,|\beta|=1\},
$$
and let
\begin{equation}\label{defgamma}
\gamma_{\alpha, \beta }(t): = \sqrt{1+|\alpha|^2}\beta t+\alpha.
\end{equation}
Clearly the family of all $\gamma_{\alpha, \beta }$ with $(\alpha,\beta)\in \mathcal{A}$ parametrizes all affine lines in $\rr^n$. Denote by $\mathcal{B}\subset \mathcal{A}$ the set of all $(\alpha,\beta)\in\mathcal{A}$ for which the line parametrized by $\gamma_{\alpha,\beta}$ intersects $X$. It is easy to see that
$\mathcal{B}$ is a compact set and 
\begin{equation}\label{estcompact}
\mathcal{B}\subset \{(\alpha,\beta)\in\mathcal{A}:|\alpha|\le R\}.
\end{equation}
It suffices to prove that for any $(\alpha,\beta)\in\mathcal{B}$ and  $N\ge \mathcal{N}(m,R,D)$ the function $f\circ \gamma_{\alpha,\beta}$ is strictly convex on  $I_{\alpha,\beta}=\{t\in\rr:\gamma_{\alpha,\beta}(t)\in X\}$. Since $X$ is a compact convex set, $I_{\alpha,\beta}$ is a compact interval or a point.

It is obvious that for $(\alpha,\beta)\in\mathcal{B}$ the set $\{t\in \rr:|\gamma_{\alpha,\beta}(t)|\le R\}$ is an interval centered at $0$ (or a point), say $[-R_{\alpha,\beta},R_{\alpha,\beta}]$. Moreover, we have 
$I_{\alpha,\beta}\subset [-R_{\alpha,\beta},R_{\alpha,\beta}]\subset [-R,R]$.

If $f$ is of the form \eqref{formoffmultidimensional1}, then we easily see that for $t\in \rr$ such that $|\gamma_{\alpha,\beta}(t)|\le R$ we have
$$
|(f\circ \gamma_{\alpha,\beta})'(t)|\le \sqrt{1+R^2}\sum_{j=1}^d \sum_{|\nu| =j}j|a_{\nu}| R^{j-1}
$$
and
$$
|(f\circ \gamma_{\alpha,\beta})''(t)|\le (1+R^2)\sum_{j=2}^d \sum_{|\nu| =j} j(j-1) |a_{\nu}|R^{j-2},
$$
so
\begin{equation}\label{estimatefprimefbis1}
|(f\circ \gamma_{\alpha,\beta})'(t)|\le D,\quad |(f\circ \gamma_{\alpha,\beta})''(t)|\le D\quad \hbox{for } t\in [-R_{\alpha,\beta},R_{\alpha,\beta}]. 
\end{equation}

A simple computation gives
\begin{equation}\label{normag}
1+ |\gamma_{\alpha, \beta }(t)|^2 = (1+|\alpha|^2)(1 +t^2),
\end{equation}
hence
$$
\vf_N\circ\gamma_{\alpha,\beta}(t)=(1+|\alpha|^2)^N (1+t^2)^Nf\circ \gamma_{\alpha,\beta}(t).
$$
Obviously $\vf_N\circ\gamma_{\alpha,\beta}$ is a  strongly
 convex function on $I_{\alpha,\beta}$ if and only if the function 
$I_{\alpha,\beta}\ni t\mapsto (1+t^2)^Nf\circ \gamma_{\alpha,\beta}(t)$ is strongly convex. 
Now applying Lemma \ref{convexonevariableoncompact} we deduce the assertion.
\end{proof}


\begin{remark}\label{thesameassertionasth9} 
Under the assumptions of Theorem \ref{convexcompactnew}, and with the same argument, we obtain the assertion of this theorem for the function $\varphi_{N,\xi}(x)=(1+|x-\xi|^2)^Nf(x)$ instead of $\varphi_N$, where $\xi\in\rr^n$, with the bound $N>\mathcal N(m,2R,D)$.
\end{remark}

\section{Convexity  at infinity}\label{convexityatinfinitysection}

We  briefly recall basic definitions. For a $C^2$ function $f$ in an open subset of~$\rr^n$,  $H_xf$ stands for the Hessian matrix of  $f$ at $x$. The associated  quadratic form $h_x:\rr^n \to  \rr$ reads
\begin{equation}\label{defh}
h_x f(y) = \langle H_x f( y),y\rangle.
\end{equation}
Recall that  the matrix 
$H_x f$ is said to be {\it positive  semidefinite} (respectively {\it   positive definite}) if  $h_x(y) \ge 0$ for any $y\in \rr^n$ (respectively  $h_x f(y) >0$ for $y\ne 0$). 
Set, for $E\subset \{1,\ldots,n\}$, $E\ne \emptyset$,
$$
\Delta^{f}_E:=\det \left[\frac{\partial^2f}{\partial x_i\partial x_j}\right]_{i,j\in E }.
$$
Recall a classical fact (Sylvester criterion):

\begin{lemat}
$H_x f$ is positive  semidefinite (respectively positive definite) if and only if  
$\Delta^{f}_E \ge 0$ (respectively $\Delta^{f}_E >0$) 
 for all nonempty $E\subset \{1,\ldots,n\}$.
\end{lemat}

Let $f\in\rr[x]$ and $n\ge 2$. 
We call $f$ \emph{locally convex} 
(respectively \emph{locally strictly convex} or \emph{locally strongly convex}) in an open set $G\subset \rr^n$ if  any point $x\in G$ has a~convex neighbourhood $U\subset \rr^n$ such that the restriction $f|_U$ is convex (respectively strictly convex or strongly convex). In particular $f$ is locally convex in $G$ if and only if $H_x f$ is positive for any $x\in G$. 
We say that $f$ is \emph{convex at infinity}  (respectively  \emph{strictly convex at infinity} or \emph{strongly convex at infinity}) if there exists $R\ge 0$ such that $f$ is locally convex (respectively locally strictly convex or locally strongly convex)  in $G=\{x\in\rr^n:|x|>R\}$. The analogous terminology will be  used for concave functions.

Let $d=\deg f\ge 0$ and let $f_0,\ldots, f_d\in\rr[x]$ be homogeneous polynomials such that $f_i=0$ or $\deg f_i=i$,
 and
$f=f_0+\cdots+f_d.$
Since $d=\deg f$, we have $f_d\ne 0$.

\begin{lemat}\label{convex2} If  $f$ is convex at infinity, then $f_d$ is a convex function.
\end{lemat}
\begin{proof}
Assume that $f_d$ is not convex. Then for some nonempty $E\subset\{1,\ldots,n\}$ and $x_0\ne 0$, 
$\Delta^{f_d}_E(x_0)<0$. Since  $\Delta^{f_d}_E$ is nonzero it must be  a homogeneous polynomial of degree $k(d-2)$, where $k$ is the number of elements of $E$. 
Then 
$$
\Delta^{f}_E(tx_0)=t^{k(d-2)}\Delta^{f_d}_E (x_0)+ F(t)
$$
with some polynomial  $F(t)$ of degree less than  $k(d-2)$. 
So $\Delta^{f}_E(tx_0)<0$ as $t\to \infty$, hence  $f$ is not convex at infinity, which contradicts the assumption.
\end{proof}

To obtain the convexity of $\vf_N$ we will assume that $f_d(x)>0$ for $x\in\rr^n\setminus\{0\}$.  This assumption is natural, as the following proposition shows. 

\begin{prop}\label{pconvex1}
The following conditions are equivalent:
\begin{enumerate}
\item[\rm (a)] $f_d(x)>0$ for $x\in\rr^n\setminus\{0\}$,

\item[\rm (b)] there exist $R>0$ and $N_0\in\nn$ such that for any integer $N\ge N_0$ the polynomial 
\begin{equation*}\label{vfpolynomial1new}
\varphi_N(x)=(1+x_1^2+\cdots+x_n^2)^Nf(x)
\end{equation*}
is locally strongly convex on $G=\{x\in\rr^n:|x|>R\}$,

\item[\rm (c)] there exists $N_0\in\nn$ such that for any integer $N\ge N_0$ the polynomial $\varphi_N$ is convex at infinity.
\end{enumerate}
\end{prop}

\begin{proof}  (a)$\Rightarrow$(b). 
We use the notations 
\eqref{defgamma} of the proof of Theorem \ref{convexcompactnew}, namely $\mathcal{A}=\{(\alpha,\beta)\in\rr^n\times\rr^n:|\beta|=1,\,\langle\alpha,\beta\rangle=0\}$ and 
$\gamma_{\alpha, \beta }(t): = \sqrt{1+|\alpha|^2}\beta t+\alpha$. 
We shall use a convenient renormalization of $f\circ \gamma_{\alpha, \beta }$.  For $(\alpha,\beta)\in\mathcal{A}$ we set
\begin{equation}\label{renorm}
g_{\alpha, \beta }(t): =  (\sqrt{1+|\alpha|^2})^{-d}f\circ \gamma_{\alpha, \beta }(t).
\end{equation}
The next crucial lemma gives an estimate on the size of the coefficients of 
$$  f\circ \gamma_{\alpha, \beta }(t)=\sum_{i=0}^d c_i(\alpha, \beta ) t^{d-i}.$$

\begin{lemat}\label{coefcomp} 
There exists a constant $C>0$ such that for any $(\alpha, \beta) \in \mathcal{A}$,
 \begin{equation}\label{coefcompeq}
|c_i(\alpha, \beta )|  \le  C(\sqrt{1+|\alpha|^2})^d \quad \hbox{for $ i=0,\dots, d$.}
\end{equation}
\end{lemat}

\begin{proof} It is enough to check the assertion for a monomial  $ax_1^{k_1}\cdots x_n^{k_n}$ with
$k_1+\cdots +k_n\le d$.
\end{proof}

Write 
$
g_{\alpha, \beta }(t)= (\sqrt{1+|\alpha|^2})^{-d}f\circ \gamma_{\alpha, \beta }(t)= \sum_{i=0}^d a_i(\alpha, \beta ) t^{d-i}.
$
Lemma  \ref{coefcomp} yields  a uniform estimate for the coefficients: 
$$
|a_i(\alpha, \beta )|\le C,\quad i= 0, \dots, d.
$$
By the assumption that $f_d (x) >0$ for $x\ne 0$ it follows that  
$$
a_0(\alpha, \beta) = f_d(\beta) \ge \inf_{|x| =1}f_d(x)=e>0,
$$
so for $K=1+2C/e$ we have
$
K\ge 1+2\sup_{(\alpha,\beta)\in A}\max_{i=1,\ldots,d}\big|\frac{a_i(\alpha,\beta)}{a_0(\alpha,\beta)}\big|^{1/i}.
$
Take $R\ge K$ and let
\begin{equation*}\label{estDeqnewold}
D\ge \max \Big\{1,\,C\sum_{i=0}^{d-1}(d-i) R^{d-i-1},
\  
C\sum_{i=0}^{d-2}(d-i)(d-i-1) 
R^{d-i-2}\Big\}.
\end{equation*}
Then $g'_{\alpha,\beta}(t)\le D$ and $g''_{\alpha,\beta}(t)\le D$ for $t\in[-R,R]$.  
Again by the assumption  that $f_d (x) >0$ for $x\ne 0$, one can assume that there exists $m>0$ such that for $|x|\ge R$ we have $f(x)\ge m(\sqrt{1+|x|^2})^d$.  So 
$$
g_{\alpha,\beta}(t)\ge (\sqrt{1+|\alpha|^2})^{-d} m\Big(\sqrt{1+|\gamma_{\alpha,\beta}(t)|^2}\Big)^d\ge m \quad \hbox{for } |\gamma_{\alpha,\beta}(t)|\ge R.
$$
To end the proof of the implication (a)$\Rightarrow$(b) 
 it is enough to apply Lemma \ref{convexonevariable}.

The implication (b)$\Rightarrow$(c) is trivial.

(c)$\Rightarrow$(a). Observe that 
\begin{equation}\label{eqfdpositive0}
f_d(x)\ge 0\quad\hbox{for } x\in\rr^n.
\end{equation}
Indeed, suppose there exists $x_0\in\rr^n\setminus\{0\}$ such that $f_d(x_0)<0$. Let $t\in\rr$, $t>0$, be such that $tx_0 \in G$. Since $f_d$ is the leading form of $f$, we may assume that $f(tx_0)<0$. Let $H\subset G$ be a compact convex neighbourhood of $tx_0$ such that $f(x)<0$ for $x\in H$. By Theorem \ref{convexcompactnew} there exists $N_0\in \nn$ such that for  $N\ge N_0$ the polynomial $\vf_N$ is strictly concave on $H$. This contradicts (c) and gives~\eqref{eqfdpositive0}.
 
Assume to the contrary that (a) fails. Then by \eqref{eqfdpositive0}, $f_d^{-1}(0)\ne \{0\}$. 
The leading form of $\vf_N$ is equal to $\psi_N(x)=(x_1^2+\cdots+x_n^2)^Nf_d(x)$ and by Lemma \ref{convex2} this form is convex. So $\psi_N^{-1}((-\infty,0])$ is a convex set, and by \eqref{eqfdpositive0}, so is $f_d^{-1}(0)=f_d^{-1}((-\infty,0])$. Consequently, the level set $f_d^{-1}(0)$ is a linear subspace, say of dimension $k>0$ (since $f_d$ is a homogeneous polynomial). By choosing a suitable coordinate system, we may assume that $\psi_N^{-1}(0)=f_d^{-1}(0)=\rr^k\times \{0\}$. Since $f_d\ne 0$, we have $k<n$. As $\psi_N|_{\rr^{k+1}\times\{0\}}$ for $k+1<n$ is also a convex function, we may assume that $n=k+1$, and moreover that $n=2$ and $k=1$. Then
$$
f_d(x_1,x_2)=x_2^{s}\tilde f(x_1,x_2)
$$
for some $s\in\nn^*$ and a homogeneous polynomial $\tilde f$ such that $\tilde f(x_1,x_2)>0$ for $(x_1,x_2)\in\rr^2\setminus \{0\}$, and $\psi_N(x_1,x_2)=x_2^s(x_1^2+x_2^2)^N\tilde f(x_1,x_2)$. Observe that for $\psi_N(x_1,x_2)=1$ we have $x_2\to 0$ as $x_1\to  \infty$ or $x_1\to-\infty$.   Consequently, the set $\psi_N^{-1}((-\infty, 1])$ is not convex, which contradicts the convexity of $\psi_N$. This gives (a) and ends the proof of (c)$\Rightarrow$(a).
The proof of  Proposition~\ref{pconvex1} is complete.
\end{proof}

\begin{twr}\label{convexonconvexsetnew}
Let $X\subset \rr^n$ be a convex closed set. Assume that 
$f$ is positive on~$X$,
\begin{equation}\label{assumptioncompact}
f_d^{-1}(0) = \{0\}
\end{equation}
and there exists $m\in\rr$ such that
\begin{equation}\label{estfcompact1}
0<m\le \inf \{f(x):x\in X\}.
\end{equation}
Then there exists $N_0\in\nn$ such that for any integer $N\ge N_0$ 
 the polynomial $\vf_N(x)=(1+x_1^2+\cdots+x_n^2)^Nf(x)$ is strongly convex on $X$.
\end{twr}

\begin{proof} If $f_d(x)<0$ for some  $x\ne 0$, then $X$ is a compact set and the assertion follows from Theorem \ref{convexcompactnew}.

Assume that $f_d(x)>0$ for any  $x\ne 0$. If $X$ is a bounded set, then the assertion immediately follows from Theorem \ref{convexcompactnew}. So assume that $X$ is  unbounded. Since $f_d(x)>0$ for $x\ne 0$, by Proposition \ref{pconvex1} there are $R\ge 0$ and $N_1\in\nn$ such that for $N\ge N_1$ the polynomial $\varphi_N$ is strongly (locally) convex in $\{x\in \rr^n:|x|\ge R\}$. By Theorem \ref{convexcompactnew} one can assume that for $N\ge N_1$, the polynomial $\vf_N$ is strongly convex on  $\{x\in X:|x|\le R+1\}$. Summing up, for $N\ge N_1$ the polynomial $\vf_N$ is strongly convex on $X$.
\end{proof}

\begin{remark}\label{reznickrem}
\rm If $X=\rr^n$ then, for any $N$ large enough, $\vf_N(x)$ 
 is not only  strictly convex, but it is a sum of squares of polynomials. More precisely,  
if $f$ satisfies the assumptions  of  Theorem
\ref{convexonconvexsetnew} with $X=\rr^n$, then its homogenization, denoted by $p$, satisfies the assumption of Reznick's theorem \cite[Theorem 3.12]{Reznick}. So after dehomogenization of 
$(x_0^2+x_1^2+\cdots+x_n^2)^Np(x)$  we see that our function $\vf_N$ is  a sum of even powers of affine functions. Hence $\vf_N$ is convex  and it is a sum of squares of polynomials.
However, this method cannot be applied  if  $X$ is a proper subset of~$\rr^n$.
\end{remark}

\begin{cor}\label{conpositivcor} Let $X\subset \rr^n$ be a closed convex 
 semialgebraic  set containing at least two points, 
 let $f\in \rr[x]$, and let  $d>\deg f$ be an even integer. Then the following conditions are equivalent:
\begin{enumerate}
\item[\rm (1)] $f$ is nonnegative on $X$,

\item[\rm (2)] for any $a,b>0$ there exists $N_0\in\nn$ such that for any integer $N\ge N_0$   the polynomial $\varphi_N(x)=(1+|x|^2)^N(f(x)+a|x|^{d}+b)$ is a strongly convex function on $X$.
\end{enumerate}
\end{cor} 

\begin{proof}
The polynomial $f(x)+a|x|^d+b$ satisfies the assumptions
of Theorem \ref{convexonconvexsetnew} if $a,b>0$. Hence the implication (1)$\Rightarrow$(2) follows from Theorem \ref{convexonconvexsetnew}. 
To prove the converse assume that  $f(x_0)<0$ for some $x_0\in \Int X$.
 Note that $X$, being convex and containing at least two points, 
   has nonempty (relative) interior. Then for sufficiently small $a,b$ and $N$ large enough, the function $-\varphi_N$ is strictly convex in a neighbourhood of $x_0$. So $\varphi_N$ is strongly concave in a neighbourhood of $x_0$, which is absurd.
\end{proof}

%

For homogeneous polynomials on $\rr^n$  we obtain the following extension of  Rez\-nick's result  mentioned in the Introduction. 
For a fixed  $f\in\rr[x]$ and a positive integer $N$, we set    $\psi_N(x):=(x_1^2+\cdots+x_n^2)^Nf(x)$.

\begin{cor}\label{Reznickconpositivcor}
Let $f\in\rr[x]$ be a nonzero homogeneous polynomial. The following conditions are equivalent:
\begin{enumerate}
\item[\rm (a)] $f(x)>0$ for $x\in\rr^n\setminus\{0\}$,

\item[\rm (b)] there exists $N_1\in\nn$ such that for any  $N\ge N_1$ the polynomial $\psi_N$ is a sum of even powers of linear functions,

\item[\rm (c)] there exists $N_2\in\nn$ such that for any  $N\ge N_2$ the polynomial
$\psi_N$ 
is a convex function,

\item[\rm (d)] there exists $N_3\in\nn$ such that for any  $N\ge N_3$ the polynomial $\psi_N$  
 is a strictly convex function.
\end{enumerate}
\end{cor}

\begin{proof} The implication (a)$\Rightarrow$(b) is  Reznick's result (see \cite[Theorem 3.12]{Reznick}). The implications (b)$\Rightarrow$(c) and (d)$\Rightarrow$(c) are trivial. The implication (c)$\Rightarrow$(a) follows by the same argument as  (c)$\Rightarrow$(a) in Proposition \ref{pconvex1}. 

To complete the proof it suffices to prove  (a)$\Rightarrow$(d). 
We will investigate  the convexity of $\psi_N$ on each line $l$ in $\rr^n$. If $0\in l$ then clearly  $\psi_N|_l$ is convex, so  
we will check the convexity of $\psi_N$ on  lines $l\subset \rr^n\setminus \{0\}$.  Since $f$ is homogeneous, it suffices to consider the convexity of $\psi_N$ on  lines of the form
$$
l=\{a+bt:t\in\rr\},\quad (a,b)\in A,
$$
where $A:=\{(a,b)\in\rr^n\times\rr^n:|a|=|b|=1,\,\langle a,b\rangle=0\}$. 
Clearly $A$ is compact.
Denote  $g(t,a,b)=f(a+bt)$ for $t\in\rr$, $(a,b)\in A$. Then 
$$
g(t,a,b)=g_0(a,b)t^d+g_1(a,b)t^{d-1}+\cdots+g_d(a,b),
$$
and $g_0(a,b)=f(b)$. So by (a) there exists $m>0$ such that $g_0(a,b)>m$ for $(a,b)\in A$. Moreover,   $f(x)\ge m$ for $x\in\rr^n$, $|x|\ge 1$, hence
$$
g(t,a,b)\ge m\quad\hbox{for } t\in\rr \hbox{ and }(a,b)\in A.
$$ 
Take $R,D\in \rr$ such that
\begin{equation*}
R\ge 1+2\max_{(a,b)\in A}\max_{1\le i\le d}\left|\frac{g_i(a,b)}{g_0(a,b)}\right|^{1/i},
\end{equation*}
and  
\begin{multline*}
D\ge \max_{(a,b)\in A}\max\Big\{1,\,\sum_{i=0}^{d-1}(d-i)|g_i(a,b)|R^{d-i-1},\\
  \sum_{i=0}^{d-2}(d-i)(d-i-1)|g_i(a,b)|R^{d-i-2}\Big\}.
\end{multline*}
Since for $(a,b)\in A$,
$$
\psi_N(a+bt)=(1+t^2)^Ng(t,a,b),
$$
 Lemma \ref{convexonevariableoncompact1} implies
   that $\psi_N$ is a strictly convex function provided  $N\ge \mathcal{N}(m,R,D)$. This gives the  implication (a)$\Rightarrow$(d) and completes the proof.
\end{proof}

\begin{remark}
Let $N_1,N_2,N_3$ be the minimal values in Corollary  \ref{Reznickconpositivcor}. Obviously $N_2\le N_1$ and $N_2\le N_3$.
It is not clear to the authors whether the equalities $N_1=N_2=N_3$ hold. 

By a result of Blekherman \cite{Blekherman}, \cite{Blekherman2}  there exist strictly convex positive forms  that are not sums of squares. 
 However, this does not answer our question, because we are interested in the smallest numbers $N_i$ such that for every $N\ge N_i$ the polynomials $\psi_N$ are respectively: sum of even powers of linear functions, convex and strictly convex. Note that multiplying a convex form by $(x_1^2+\cdots+x_n^2)^N$ may produce a~~nonconvex form.
 
 For instance the polynomial $f(x,y)=(x-ky)^2+y^2$ is a strictly convex sum of squares of linear forms. However  for sufficiently large $k$  we can find $N  \ge 1$ such that the polynomial $\psi_N$ is not convex (cf. Example \ref{exa0}) and consequently not a~~sum of even powers of linear functions. 
 
\end{remark}

\section{A  proximity algorithm for a polynomial on a convex set}

Let $X\subset \rr^n$ be a compact convex  semialgebraic set. We consider a polynomial  $f$ restricted to $X$.
 We  propose an algorithm, based 
on our convexification   method, which produces  a  sequence converging to a critical point of $f$ on $X$.

Using a translation and   a dilatation  we may assume that $X$
is contained in a ball of radius $1/2$.
Replacing $f$ by $f+c$,  where $c$ is a constant large enough we may assume that 
$m =  \inf\{f(x):x\in X \}= D>0$, where $D$ is a bound for the  absolute value of the first and the second directional derivatives
of $f$ (along vectors  of  norm~$1$). Indeed, we may increase $D$ in such a way that $|f(x)|\le D$ for $x\in X$. Then we put $c=2D$,  hence  $f(x)+c\ge D$ for $x\in X$. 
Since now $m=D$ and $2R=1$, by \eqref{formulanaN} we have $ \mathcal{N}(m,2R,D) = 6 $.

By Remark \ref{thesameassertionasth9}, with  $N=6$ and some   $\mu >0$  
the function 
$$\varphi_{N,\xi}(x):=(1+|x-\xi|^2)^Nf(x)$$
is $\mu$-strongly convex on $X$  for any  $\xi\in X$. This means that 
\begin{equation}\label{eqstronglyconv}
\vf_{N,\xi}(y)\ge \vf_{N,\xi}(x)+\langle y-x,\grad \vf_{N,\xi}(x)\rangle+\frac{\mu}{2}|y-x|^2\quad\hbox{for $x,y\in X$.}
\end{equation}
Recall that any strictly convex, hence in particular any strongly convex, function $\varphi$ on a  convex closed set  $X$ admits a unique point, denoted by $\hbox{argmin}_X \,   \varphi$, at which $\varphi$ attains its minimum on $X$.

Choose  an arbitrary point $a_0 \in X$, and by induction set 
\begin{equation}\label{eqdeciag}
a_\nu:= \hbox{argmin}_X \,  \varphi_{N,{a_{\nu-1}}}.
\end{equation}

\begin{lemat}\label{corextra}
For any $\nu\in\nn$ we have
$$
|a_{\nu+1}-a_\nu|=\dist (a_\nu,f^{-1}(f(a_{\nu+1}))\cap X).
$$
\end{lemat}

\begin{proof}
If  $|a'-a_\nu|<|a_{\nu+1}-a_\nu|$ for some $a'\in f^{-1}(f(a_{\nu+1}))\cap X$, then 
by the definition of $\varphi_{N,a_\nu}$ we have $\varphi_{N,a_\nu}(a')<\varphi_{N,a_\nu}(a_{\nu+1})$, which contradicts the definition of $a_{\nu+1}$. So, $|a_{\nu+1}-a_\nu|\le \dist (a_\nu,f^{-1}(f(a_{\nu+1}))\cap X)$. The opposite inequality is obvious.
\end{proof}

\begin{lemat}\label{cor2extra}
For any $\nu\in\nn$ we have
$$
f(a_{\nu+1})\le \frac{f(a_\nu)-\frac{\mu}{2}|a_{\nu+1}-a_\nu|^2}{(1+|a_{\nu+1}-a_\nu|^2)^N}.
$$
In particular the sequence $f(a_\nu)$ is decreasing.
\end{lemat}

\begin{proof}
Since  $\varphi_{N,a_\nu}$ is strongly convex,  the definition of $a_{\nu+1}$ implies that the function 
$$
[0,1]\ni t\mapsto \varphi_{N,a_\nu}(a_\nu+t(a_{\nu+1}-a_\nu))
$$
decreases, so $
\langle a_{\nu+1}-a_\nu,\grad \varphi_{N,a_\nu}(a_{\nu+1})\rangle\le 0$. 
Thus, by \eqref{eqstronglyconv} we see that 
$$
\vf_{N,a_\nu}(a_\nu)\ge \vf_{N,a_\nu}(a_{\nu+1})+\frac{\mu}{2}|a_{\nu}-a_{\nu+1}|^2.
$$
Again, by the definition of $\vf_{N,a_\nu}$ we have
$$
f(a_\nu)\ge (1+|a_{\nu+1}-a_\nu|^2)^Nf(a_{\nu+1})+\frac{\mu}{2}|a_{\nu+1}-a_\nu|^2.
$$
This ends the proof of the lemma.
\end{proof}
Now we estimate from below the length of $|a_{\nu}-a_{\nu+1}|$, i.e., of  the  step in our  sequence. It is enough to consider
only the one-dimensional case with $a_\nu = 0$. By a  direct computation we obtain:

\begin{lemat}\label{lemstep}
Let $f:[0,{{}\eta}] \to \rr$  be a $C^1$ function {{} such} that ${{}0<} f\le C$ and  $f' \le -\eta$ on $[0,\eta]$ 
for some {{}$C\ge\frac{1}{2}$ and} $\eta >0$.
Assume that $\varphi_N (x) = (1+x^2)^Nf(x)$ is strictly convex  on $[0,\eta]$. Then 
$b_1: =\textup{argmin}_{{{}[0,\eta]}} \,  \varphi_N  \ge \frac{\eta}{2NC}$. Hence  $f(0) - f(b_1) \ge  \frac{\eta^2}{2NC}$.
\end{lemat}
Let $f$ be a  $C^1$ function in a neighborhood $U$ of a closed set  $X \subset \rr^n$. Recall that  $a\in X$ is  a \emph{{{}lower} critical point 
of $f$ on $X$} if 
\begin{equation}\label{defcritical}
\langle \nabla f(a), x-a\rangle \ge 0 \quad\hbox{for  }x\in X\text{ {} in a neighbourhood of }{{}a}.
\end{equation}

We denote by $\Sigma_X f$ the set of {{}lower} critical points of $f$ on $X$, and by $\Sigma f: =\{x\in U:  \nabla f(x)= 0\}$ the set of ordinary critical points of $f$. The following proposition recalls all the necessary properties of these sets.

\begin{prop}\label{pconvexcrit} Assume that $X \subset \rr^n$ is closed  
 and $f:\rr^n\to \rr$ is a $C^1$ function. Then:
\begin{enumerate}
\item $ X \cap \Sigma f \subset \Sigma_X f$;
\item if $f$ restricted to $X$ has a local minimum at $a$, then $a\in \Sigma_X f$;
\item if $M\subset X$ is a smooth manifold and $a\in M\cap \Sigma_X f$, then for any $z \in T_aM$,
$$
\langle \nabla f(a), z \rangle  =0;
$$
\item if $f$ is a polynomial and  $X$ is semialgebraic, then $\Sigma_X f$ is a semialgebraic set and 
$ f(\Sigma_X f)$ is a finite set.

\end{enumerate}
\end{prop}

\begin{proof} The first three statements follow immediately from the definition. If  $f$ is a  polynomial and $X$ is  semialgebraic then the set $\Sigma_X f$ is described by a first order formula (in the language of ordered fields) so it is 
 semialgebraic as well (see e.g. \cite[Chapter 2]{BCRoy}). Hence $\Sigma_X f$ has finitely many connected components (in  fact connected by piecewise $C^1$ semialgebraic
 arcs).  Each such component is a finite union of smooth manifolds, hence by condition (3) the function $f$
 is constant on it. So  $ f(\Sigma_X f)$ is a finite set.
 \end{proof}

\begin{twr}\label{tw1extra}  Let $X\subset \rr^n$ be a compact convex  semialgebraic set and $f:\rr^n \to \rr$ a {{}positive} polynomial  {{}on $X$}.
Let $ a_\nu$ be the sequence defined by \eqref{eqdeciag} with $a_0\in X$. 
Then the limit  
$$
a^*=\lim_{\nu\to\infty}a_\nu
$$ 
exists, and $a^* \in \Sigma_X f$.
\end{twr}

\begin{remark}\label{tw1extrarem}
Note that  Lemmas \ref{corextra} and \ref{cor2extra} hold true for  any function of class  $C^2$ in a neighborhood of $X$. However, they are not sufficient to prove the convergence of the 
sequence $a_\nu$ (at least we have not been able to do this). For  $f$ polynomial the convergence of the 
sequence $a_\nu$  will follow from  some fine properties of the gradient trajectories of  polynomials.
\end{remark}
\noindent 
{\it Proof of Theorem \ref{tw1extra}.}
First, assuming that  $a^*=\lim_{\nu\to\infty}a_\nu$ exists, we shall prove that $a^* \in \Sigma_X f$. Suppose   that $a^* \notin \Sigma_X f$, so there exists $x\in X$ with $\langle \nabla f(a^*), x-a^*\rangle < 0$. Then there exists $\eta>0$ such that $\langle \nabla f(a^*+t(x-a^*)), x-a^*\rangle < -\eta$ for $t\in[0,\eta]$. 
By  continuity  the same holds with $a^*$ replaced by $a_\nu$
 for $a_\nu$ sufficiently close to $a^*$. Moreover, we may assume that $|f(a_\nu)-f(a^*)|< \frac{\eta^2}{2NC}$, where $C\ge f(x)$ for $x\in X$. Hence  
by continuity and  Lemma \ref{lemstep} we obtain $f(a_{\nu+1}) < f(a^*)$, which is a contradiction.

Recall  now the Comparison Principle \cite[Lemma 4.2]{DK}. Let $f:\rr^n\to \rr$ be a polynomial and  let 
$M\subset \rr^n$ be a smooth bounded semialgebraic  set. Let $\grad f(x)$ denote the gradient  of $f$ with respect to  the standard Euclidean scalar product,
and  $\grad_M f(x)$  its projection on $T_x M$,  the tangent space to $M$ at $x$.

 Let $\Gamma_M \subset  \overline M$ be a semialgebraic curve meeting each level set of $f$ and such that for 
every point $y\in\Gamma_{{}M}$ we have $|\grad_M f(y)|\le |\grad_M f(x)|$ for all $x \in f^{-1}(f(y))\cap {{}\overline{M}}$. By standard arguments 
(semialgebraic choice) such a curve always exists; it is called a 
{\it talweg} or  {\it a ridge-valley line of $f$ in $X$}.
Then the
following lemma holds.

\begin{lemat}[Comparison Principle]\label{comparisonpronciple} For every pair of values $a < b$ taken by $f$, the length of any trajectory of $\grad_{{}M} f$ lying in $f^{-1}((a, b))\cap M$ is bounded by the
length of $\Gamma_M \cap f^{-1}((a, b))$.
\end{lemat}
To prove that $\lim_{\nu\to\infty}a_\nu$  exists recall  first that
by Lemma \ref{cor2extra} we have $$f(a_\nu )\ge f(a_{\nu +1})\ge \cdots  \ge f_*:=\lim_{\nu\to\infty}f(a_\nu).
$$
By Proposition \ref{pconvexcrit}$(4)$ the set $ f(\Sigma_X f)$ of critical values of $f$ on $X$ is finite, so we may assume that either the sequence $f(a_\nu )$  is 
eventually constant,  
or $(f(a_\nu), f_*) \cap  f(\Sigma_X f) =\emptyset $
 for $\nu$ large enough.
Clearly in the first case, by Lemma \ref{cor2extra},  also  the sequence $a_\nu$ is eventually constant.
So we assume from now on that the sequence $f(a_\nu )$ is strictly decreasing and 
$(f(a_0), f_*) \cap  f(\Sigma_X f) =\emptyset $.

The set $X$ is semialgebraic, so there exists a stratification $X=\bigcup_{i\in I} M_i$, i.e., a finite disjoint union of connected smooth semialgebraic sets, called {\it strata}. Moreover $\overline M_i \setminus M_i$ is a union of some of the $M_j$'s of dimension smaller than $\dim M_i$ (cf. \cite[Chapter 2]{beri}).
We can refine this stratification in such a way that
 $f$ is of constant rank on each $ M_i$, $i\in I$; then our polynomial $f$ restricted
 to $ M_i$ is either a constant or  a submersion. Let  $I^* = \{i\in I: \, \hbox{rank}\, f|_{ M_i} = 1\}$; note 
that  $C_X f = \bigcup_{i\in I\setminus I^*} f(M_i)$  is a finite set. Since the sequence $f(a_\nu )$ is strictly decreasing  we may assume that
$(f(a_0), f_*) \cap  C_X f =\emptyset $.

To each $M_i$, $i\in I^*$, we can associate a semialgebraic curve $\Gamma_i: = \Gamma_{M_i}$ which is  a  talweg of $f$ in
$M_i$. Set $\Gamma := \bigcup_{i\in I^*} \Gamma_{i}$. 

Recall that,  by  Lemma \ref{corextra},  $a_{\nu +1}$ is the point closest  to $a_{\nu}$  on the fiber $f^{-1}(f( a_{\nu + 1}))\allowbreak\cap X$. To estimate  $|a_{\nu+1}-a_\nu|$ we will construct a continuous curve 
 $\gamma_\nu : [t_\nu, t_{\nu + 1}] \to  X$  such that $\gamma_\nu(t_\nu) = a_\nu$ and 
$f({{}\gamma_\nu( t_{\nu + 1})}) = f( a_{\nu + 1}) $.
By Lemma \ref{corextra} we will then have 
$|a_{\nu+1}-a_\nu| \le \length(\gamma_\nu)$.
The curve  $\gamma_\nu$ will be  a piecewise  trajectory of  $-\nabla_{M_i}f$ (more precisely, of $-\nabla_{M_i}f/\vert \nabla_{M_i} f \vert$).
Hence, by the Comparison Principle,
$$|a_{\nu+1}-a_\nu| \le \length(\gamma_\nu)
\le  \length(\Gamma \cap f^{-1}((f( a_{\nu + 1}),f(a_\nu)))).
 $$
Recall that $\Gamma$, being a bounded semialgebraic curve,  has finite length (see e.g. 
\cite[Corollary  5.2]{comteyomdin}); therefore 
$$
\sum_{\nu=0}^\infty |a_{\nu+1}-a_\nu| \le \length(\Gamma \cap f^{-1}((f_*,f(a_0))))<\infty.
$$
So the series $\sum_{\nu=0}^\infty |a_{\nu+1}-a_\nu|$ is convergent, which implies  that $a^*=\lim_{\nu\to\infty}a_\nu$ exists.

{\it Construction of the curve $ \gamma_\nu. $}
Assume that $a_\nu$ belongs to a stratum $M_i$ for some $i\in I^*$. Let $\gamma_\nu : [t_\nu, t_\nu^1) \to M_i$  be a  trajectory 
of  $V_i:=-\nabla_{M_i} f/\vert \nabla_{M_i} f \vert$. By a trajectory we mean a maximal  solution  (to the right) 
of $\gamma'=V_i$ in $M_i$. Note that 
 $$b_1^*=\lim_{s \nearrow  t_\nu^1}\gamma_\nu(s) \in \overline M_i \setminus M_i $$ 
exists.
 Indeed, by Lemma \ref{comparisonpronciple} any maximal trajectory of $V_i$ has  finite length so it has a limit in  $\overline M_i$. But the vector field $V_i$ does not vanish on $M_i$, hence this limit belongs to 
 $\overline M_i \setminus M_i $, which is a union of strata of smaller dimension.
 
If   $f(b_1^*) \le  f(a_{\nu + 1})$ then there exists $ t_{\nu + 1} \in  [t_\nu, t_\nu^1]$ such that 
$f({{}\gamma_\nu( t_{\nu + 1})}) = f( a_{\nu + 1}) $, so   $\gamma_\nu$  restricted to  $[t_\nu, t_{\nu + 1}] $ is the curve we are looking for.
 Now if  $f(b_1^*)>   f(a_{\nu + 1})$, then $b_1^* \in M_{i_1}$ for some $i_1\in I^*$ such that 
 $\dim  M_{i_1} < \dim  M_{i}$. We repeat the above construction  on $M_{i_1}$ starting from the point
 $b_1^*$, then  we glue it  with the previous one. In this way  the dimension of the stratum in which our curve 
 $\gamma_\nu$ stays  is strictly  dec{{}r}easing, but this dimension is always at least~$1$.
 Finally we will reach the level $f^{-1}(a_{\nu + 1})$. 
 Indeed, when our curve arrives at a point in a stratum of dimension $1$ we follow this stratum until we  arrive  at the level
 $f^{-1}(a_{\nu + 1})$ since $(f(a_0), f_*) \cap  C_X f =\emptyset $. The estimate 
 $$ \length(\gamma_\nu)
\le 
\length 
(\Gamma \cap f^{-1}(f( a_{\nu + 1}{{})},f(a_\nu)))
 $$
 follows from  Comparison Principle.\hfill$\square$

 \begin{remark} In the case when $X$ is  a closed ball (or more generally  when $X$  has a smooth boundary)
 the length of the curve $\Gamma$ can be  effectively estimated (see \cite{DK}).
 \end{remark}

\section*{Acknowledgements} 
We are  grateful to anonymous  referees  and to J\'er\^ome Bolte for their valuable comments and suggestions which led to a
substantial improvement of this paper.

\bigskip

\noindent Krzysztof Kurdyka \newline 
Universit\'e de Savoie\newline
Laboratoire de Math\'ematiques (LAMA) UMR-5127 de CNRS\newline
73-376 Le Bourget-du-Lac Cedex, FRANCE\newline
\indent E-mail: Krzysztof.Kurdyka@univ-savoie.fr

\medskip
\noindent Stanis\l{}aw Spodzieja\newline 
Faculty of Mathematics and Computer Science, University of \L \'od\'z, \newline
S. Banacha 22, 90-238 \L \'od\'z, POLAND \newline
 \indent E-mail: spodziej@math.uni.lodz.pl

\end{document}